\documentclass[final,3p,times]{elsarticle}

\usepackage{graphicx}
\usepackage{amssymb}
\usepackage{amsmath}
\usepackage{commath}
\usepackage{amsfonts}
\usepackage{graphicx,verbatim}
\usepackage{booktabs} 
\usepackage{array}
\usepackage{amsthm}
\usepackage{units}

\newcommand{\fP}{\mathfrak{P}}

\newcommand{\cI}{\mathcal{I}}
\newcommand{\cO}{\mathcal{O}}

\newcommand{\valp}{\mbox{val}_{p}}
\newcommand{\valpi}{\mbox{val}_{\pi}}
\newcommand{\etap}{\eta_{p}}

\newtheorem{theorem}[subsection]{Theorem}
\newtheorem{definition}[subsection]{Definition}
\newtheorem{proposition}[subsection]{Proposition}
\newtheorem{corollary}[subsection]{Corollary}

\newtheorem{remark}[subsection]{Remark}

\journal{ARXIV}

\begin{document}

\begin{frontmatter}


\title{Zeta Functions of Certain Quadratic Orders}
\author{Malors Espinosa}
\ead{srolam.espinosalara@mail.utoronto.ca, Department of Mathematics, University of Toronto
Room 6290, 40 St. George Street, Toronto, Ontario, Canada M5S 2E4}


\begin{abstract}
In \cite{LanBE04} Langlands provides a formula for certain product of orbital integrals in $GL(2, \mathbb{Q})$. Its generalization has become an important question for the strategy of Beyond Endoscopy. Arthur in \cite{ARTHUR2018425} predicts this formula should coincide with a product of polynomials associated to zeta functions of orders constructed in \cite{ZYun} by Zhiwei Yun. In this paper we compute, for a certain family of orders, explicit formulas for these zeta functions by a recursive method. We use these zeta functions in \cite{malors21} to prove that Arthur's prediction is correct.  

\end{abstract}
\end{frontmatter}


\section{Introduction and Results of this Article}

Let $K$ be a $p-$adic field and $\cO_K$ be its ring of integers. Let $L$ be a reduced $K$-algebra of dimension $2$ over $K$ and $\cO_L$ the integral closure of $\cO_K$ in $L$. In \cite{ZYun}, Yun defines for a given order $\cO\subseteq L$ a zeta function $\zeta_{\cO}$. It generalizes the classical construction of the zeta function of $\cO_L$. In \cite{ZYun} it is also proven that $\zeta_{\cO}$ is a rational function, that is
\begin{equation*}
    \zeta_{\cO}(s) = \dfrac{P(q^{-s})}{V(s)},
\end{equation*}
where $q$ is the cardinality of the residue field of $K$. $V$ is a polynomial in $q^{-s}$ that only depends on the extension $L/K$, while $P$ is a polynomial that depends on $\cO$.

The general theory of discrete valuation rings guarantees the existence of $\Delta\in\cO_L$ such that
\begin{equation*}
    \cO_L = \cO_K[\Delta].
\end{equation*}
Using this, we define for each integer $n\ge 0$, the order 
\begin{equation*}
    \cO_n := \cO_K[p^n\Delta],
\end{equation*}
where $p$ is the uniformizer of $K$. This is an order and thus has associated a zeta function $\zeta_n$. The main result of this article is the explicit computation of the polynomial $P$ for $\zeta_n$ (i.e. the numerator of its rational expression). More concretely, our main theorem is:
\begin{theorem}(Theorem \ref{solutionrecurrence} in section \ref{solutionoftherecurrence})
For each $n\ge 0$ define the following polynomials:
\begin{align*}
        R_n(X) &= 1 + qX^2 + q^2X^4 + ... + q^nX^{2n},
\end{align*}
and for $n\ge 1$ define
\begin{align*}
    U_n(X) &= (1 + X)R_{n - 1}(X) + q^nX^{2n},\\
    S_n(X) &= (1 - X)R_{n - 1}(X) + q^nX^{2n}.
\end{align*}
Finally, also put $U_0(X) = S_0(X) = 1$. Explicitly, these polynomials are
\begin{align*}
    R_n(X) &= 1 + qX^2 + q^2X^4 + ... + q^nX^{2n},\\
    U_n(X) &= 1 + X + qX^2 + qX^3 + ... + q^{n - 1}X^{2n - 2} + q^{n - 1}X^{2n - 1} + q^nX^{2n},\\
    S_n(X) &= 1 - X + qX^2 - qX^3 + ... + q^{n - 1}X^{2n - 2} - q^{n - 1}X^{2n - 1} + q^nX^{2n}.
\end{align*}Then the solution of the ramified, unramified and split case recurrence, respectively, satisfy
\begin{align*}
    (1 - q^{-s})\zeta_n(s) &= R_n(q^{-s}),\\
    (1 - q^{-2s})\zeta_n(s) &= U_n(q^{-s}),\\
    (1 - q^{-s})^2\zeta_n(s) &= S_n(q^{-s}).
\end{align*}
\end{theorem}
Notice that the assumption that $L$ has dimension $2$ over $K$ implies that $L$ is one of three possible options: a ramified field extension of $K$, an unramified field extension of $K$ or $L = K\times K$. 

These polynomials were already computed in \cite{Kaneko2003} (see the proof of Theorem 1 there). In that paper these polynomials were computed by factoring into Euler products a global formula deduced by Zagier in \cite{Zagier} (see proposition 3(iii), page 130 there) and more generally in the recent paper \cite{KMizuno} (see their Theorem 1). In particular, it relies entirely on global methods. On the contrary, we compute these polynomials by a recursive method that is entirely local in nature and that does not require any global formula beforehand. Concretely, the result is the following 
\begin{theorem}(Theorem \ref{recurrencestatement} in section \ref{studyoftheprincipalpart})
For $n\ge 1$, the zeta functions of the orders $\cO_n$ satisfy the recurrence relation
\begin{equation*}
    \zeta_n(s) = \zeta_n^P(s) + q^{-s}\zeta_{n-1}(s).
\end{equation*}
\end{theorem}
In this theorem, $\zeta_n^P(s)$ is the contribution to the zeta function $\zeta_n$ from the principal ideals of $\cO_n$. Explicitly it is
\begin{equation*}
    \zeta_n^P(s) = \displaystyle\sum_{I\subseteq\cO_n}\dfrac{1}{[\cO_n:I]^s},
\end{equation*}
where the sum runs over the principal ideals of finite index in $\cO_n$. 

The recurrence relation is the same regardless of the type of extension $L/K$ is. This obfuscates an important point: when $L = K\times K$, we have that $\cO_L = \cO_K\times\cO_K$ is not a discrete valuation ring nor a local ring nor an integral domain. We will be able to overcome this in every occasion, but a significant part of the work done in this paper is needed precisely to deal with this case. 

The recurrence relation follows from a dichotomy that ideals $I\subseteq\cO_n$ of $\cO_K-$rank 2 satisfy: \textit{they are principal ideals of $\cO_n$ or they are the image of an ideal of $\cO_{n - 1}$ under multiplication by $p$}. The recursive relation is the algebraic restatement of this fact. 

To prove this dichotomy we introduce the concept of representative of an ideal $I$. It is an element which has minimal possible valuation, in each coordinate simultaneously, among the elements of $I$. The existence of such elements is trivial when $L/K$ is a field extension, but requires care when $L = K\times K$. Once we know they exist, we can prove the following result, which implies the dichotomy:
\begin{proposition}(Proposition \ref{reprentatives} in section \ref{sectionarithmetic})
Let $n\ge 0$. For every rank $2$-ideal $I\subseteq \cO_n$ and every representative $x\in I$, there exists $0\le i \le n$ such that
\begin{equation*}
    I = x\cO_i.
\end{equation*}
Furthermore, $i$ only depends on $I$. 
\end{proposition}
This proposition follows from standard arguments when $L/K$ is a field extension due to the existence of a valuation in $\cO_L$. We introduce the notion of representative to adapt those argument when $L = K\times K$. 

In order to solve the recurrence we need the initial condition $\zeta_0$ and $\zeta_n^P(s)$ explicitly. The initial condition is straightforward because it is the zeta function of $\cO_L$. For $\zeta_n^P$ we prove:
\begin{proposition}(Proposition \ref{recurrenceshape} in section \ref{studyoftheprincipalpart})
The principal part of the zeta functions of the orders $\cO_n$ satisfies
\begin{equation*}
        \zeta_n^P(s) = \displaystyle\sum_{d = 0}^{n - 1}\dfrac{[\cO_{0}^*:\cO_n^*]}{[\cO_{0}^*:\cO_{n - d}^*]}\cdot\dfrac{1}{q^{2ds}} + \dfrac{[\cO_{0}^*:\cO_n^*]}{q^{2ns}V(s)}.
\end{equation*}
Recalll that $V(s)$ is the denominator of the rational expression of $\zeta_n(s)$. It is an explicit polynomial that depends only on $L/K$. 
\end{proposition}
The proof of this proposition follows from understanding how different principal ideals contribute to $\zeta_n^P(s)$. We divide principal ideals of $\cO_n$ into two classes, according to the properties of their representatives (vaguely, we divide them according to whether the representative is large or small in a specific sense). Each of the classes is responsible for one of the two terms in the expression of $\zeta_n^P(s)$. 

After the expansion of $\zeta_n^p(s)$ is settled, we prove
\begin{proposition}(Proposition \ref{valuesindices} in section \ref{solutionoftherecurrence})
The indices of the units subgroups satisfy for $n\ge 1$,
\begin{equation*}
    [\cO_0^*:\cO_n^*] = \left\{
    \begin{array}{ll}
        q^n & \mbox{in the ramified case,}\\
        (q+1)q^{n-1} & \mbox{in the unramified case,}\\
        (q-1)q^{n-1} & \mbox{in the split case.}\\
    \end{array}
    \right.
\end{equation*}
And, of course, $[\cO_0^*:\cO_0^*] = 1$.
\end{proposition}
The fact that the indices and the initial condition $\zeta_0(s)$ vary according to $L/K$ is the responsible of the existence of three different families of polynomials. The recurrence is solved with this information by a straightforward induction.

The source of motivation of the author to find these polynomials is the Langlands Program. 
The central conjecture of the Langlands Program is the Principle of Functoriality. This important conjecture remains unproven despite it being known in several cases which have remarkable implications in number theory, representation theory and harmonic analysis.

In \cite{LanBE04}, Langlands suggests a new approach to try to prove the principle of Functoriality which has come to be known as Beyond Endoscopy. Vaguely speaking, it is a strategy to construct a new trace formula which would allow a more careful analysis of the automorphic $L-$functions.

In \cite{LanBE04}, a particular formula for the product of certain orbital integrals for each of the local completions of $\mbox{GL}(2, \mathbb{Q})$ is used in a prominent role for the whole analysis. It is not clear how to generalize this formula to $\mbox{GL}(n, \mathbb{Q})$ for $n>2$. 

In \cite{ARTHUR2018425}, Arthur conjectures that such a generalization might be available in the work \cite{ZYun} of Zhiwei Yun. Hence, it becomes relevant to understand if the formula used by Langlands in \cite{LanBE04} for the product of those orbital integrals in $\mbox{GL}(2, \mathbb{Q})$ can be deduced from the zeta functions of orders in the way Arthur predicts. This turns out to be the case and the explanation of this is the topic of \cite{malors21}. The first step for the verification of the conjecture of Arthur is the explicit computations of the polynomials of orders that we have explained above. 

Furthermore, \cite{LanBE04} and \cite{ARTHUR2018425} work over $\mathbb{Q}$. Here we work over general $p-$adic fields and their quadratic extensions, which has as consequence that the formula of Langlands can be generalized for general algebraic number fields. We develop this formula in \cite{malors21} based on the work of this paper.

A more careful discussion of the literature is in order now. Similar Zeta functions associated with orders were defined and studied by Solomon in \cite{Solomon1977ZetaFA}. Their properties were furthermore understood and developed in the important articles \cite{Reiner1}, \cite{Reiner3}, \cite{Reiner2} by Bushnell and Reiner. These zeta functions differ from those we discuss in this paper in that they count $\cO_K$-lattices of rank $2$ inside of $\cO_n$, while the ones defined by Yun in \cite{ZYun} count ideals (after some simplifications explained in section 2). As we have mentioned, the interest of our investigations stems from Beyond Endoscopy and its specific relationship with the work of Yun (see \cite{ZYun}, section 4). Hence, the need to study the ideal counting zeta functions as we do in this paper.

In \cite{Promode} a recurrence relation for the zeta functions of lattices of the same sequence of orders is deduced (see equation (32) in \cite{Promode}). That recurrence relation is similar to the one we obtain but not the same. They deduce their recurrence relation by integral methods while we do it by isolating the principal ideals and organizing them by their \textit{type}. This allows to construct explicitly all ideals of one of the given orders $\cO_n$ by following the recursive method. At the moment of writing, whether one recurrence can be deduced from the other is unclear to the author and remains an interesting question. Finally, similar manipulations as the ones we do to count ideals of each type appear in \cite{ibukiyama2023genus}. 

Let us now conclude with a short overview of the organization of this paper. In section \ref{sectionzetafunctionsoforders} we review the results of \cite{ZYun} that we will need. In section \ref{sectionarithmetic} we prove the dichotomy mentioned above that ideals satisfy. In section \ref{studyoftheprincipalpart} we prove the expansion of the principal part $\zeta_n^P(s)$. Finally, in section \ref{solutionoftherecurrence} we compute the indices of the unit subgroups and solve the recurrence.

\section{Zeta Functions of Orders}\label{sectionzetafunctionsoforders}

In this article, $K$ will always be a $p-$adic field with uniformizer $p$. We let $q$ be the cardinality of its residue field. $\cO_K$ will denote its ring of integers, $\valp$ its $p$-adic valuation.

In \cite{ZYun}, Yun defines zeta functions associated to orders within finite dimensional reduced $K$-algebras. In this section we will discuss the results of \cite{ZYun} that we need suited out for our purposes. 

Let $L$ be a finite dimensional reduced $K-$algebra and $\cO_L$ the integral closure of $\cO_K$ in $L$. We now define the orders associated to the extension $L/K$.
\begin{definition}
An \textbf{order} $\cO\subseteq L$ is a finitely generated $\cO_K$-module such that
\begin{equation*}
    \cO\otimes_{\cO_K} K = L. 
\end{equation*}
An order $\cO$ is \textbf{monogenic} if there exists an element $\Delta\in L$ such that
\begin{equation*}
    \cO = \cO_K[\Delta].
\end{equation*}
\end{definition}
\begin{remark}
    Classically, orders are only defined for the case when $L/K$ is a field extension. This definition extends the standard one, but several results of the classical theory have to be shown in a different way because $\cO_L$ is no longer a discrete valuation ring nor a local ring. Monogenicity will allow us to recover many of these results.
\end{remark}

\begin{definition}\label{zetafunction}
Let $\cO\subseteq \cO_L$ be a monogenic order. We define its \textbf{zeta function} as
\begin{equation*}
    \zeta_{\cO}(s) = \displaystyle\sum_{I\subseteq\cO} \dfrac{1}{[\cO : I]^s},
\end{equation*}
where the sum runs over all ideals $I\subseteq\cO$ of finite index. 
\end{definition}
 \begin{remark}\label{isomorphictodual}
Let $\cO\subseteq\cO_L$ be a general order, not necessarily monogenic. Its dual order is defined as
\begin{equation*}
    \check{\cO} = \left\{x\in L \mid \langle x, y \rangle \in \cO_K \mbox{ for all } y\in \cO\right\},
\end{equation*}
where $\langle \cdot, \cdot\rangle$ is a pairing that makes $\cO_L$ self dual. In this context, Yun defines the zeta function as
\begin{equation*}
      \zeta_{\cO}(s) = \displaystyle\sum_{I\subseteq\check{\cO}} \dfrac{1}{[\cO : I]^s},
\end{equation*}
where the sum runs over all $\cO-$submodules of $\check{\cO}$ of finite index. When $\cO$ and $\check{\cO}$ are isomorphic as $\cO-$modules both definitions of the zeta functions coincide. Monogenic orders satisfy this, which explains why we work with definition \ref{zetafunction}. We refer the reader to \cite{ZYun} for a more detailed discussion.
\end{remark}

It is well known that $L$ can be written as a product of $g$ fields
\begin{equation*}
    L = L_1\times...\times L_g, 
\end{equation*}
for some positive integer $g$. For each $i = 1,..., g$, the residue field of $L_i$ has cardinality $Q_i = q^{f_i}$ and the prime ideal $\pi_i\cO_{L_i}$ of $\cO_{L_i}$ satisfies 
\begin{equation*}
    (\pi_i \cO_{L_i})^{e_i} = p\cO_{L_i},
\end{equation*}
where $\pi_i$ is a uniformizer of $\cO_{L_i}$. We define a function of a complex parameter $s$ as 
\begin{align*}
    V(s) &= \displaystyle\prod_{i = 1}^g(1 - q^{-f_is}),
\end{align*}
Let $\cO\subseteq\cO_L$ be a monogenic order. Define 
\begin{equation*}
    n:= \mbox{length}_{\cO_K}(\cO_L/\cO) = \log_q([\cO_L: \cO]).
\end{equation*}
 Notice that $V(s)$ does not depend on the order $\cO$ selected. Now we can state the following
\begin{proposition}\label{polynomial}
 In the context of the previous discussion, define the function
\begin{equation*}
    \tilde{J}(s) = q^{ns}V(s)\zeta_{\cO}(s).
\end{equation*}
Then there exists a polynomial $P(x)\in 1 + x\mathbb{Z}[x]$ of degree $2n$ such that
\begin{equation*}
    V(s)\zeta_{\cO}(s) = P(q^{-s}).
\end{equation*}
Furthermore, $\tilde{J}$ satisfies the functional equation
\begin{equation*}
    \tilde{J}(s) = \tilde{J}(1 - s),
\end{equation*}
or equivalently, the polynomial $P$ satisfies
\begin{equation*}
    (qx^2)^{n}P\left(\dfrac{1}{qx}\right) = P(x).
\end{equation*}
This is proved in \cite{ZYun} as theorem 2.5.
\end{proposition}

 The polynomials of this proposition are the ones we will find, for a particular set of orders, in subsequent sections of this paper.


\section{Arithmetic of the Orders $\cO_n$}\label{sectionarithmetic}

From now on we suppose $L$ is a reduced $K-$algebra of dimension $2$ over $K$. This implies $L$ is a quadratic extension of $K$ or $L = K\times K$. In the latter case, we always consider $K$ embedded diagonally in $L$. 

\begin{definition}
In the context of the above discussion, we refer to the case $L = K \times K$ as the \textbf{split case}. On the other hand, when $L$ is a quadratic field extension of $K$ we refer to the case as the \textbf{nonsplit case}.  
\end{definition}
When we do not distinguish between the split and nonsplit cases it is because the definitions or arguments make sense for both cases.

Recall that $\cO_L$ is the integral closure of $\cO_K$ in $L$. In the nonsplit case this is the ring of integers of $L$, while in the split case it coincides with $\cO_K\times \cO_K$. Furthermore, since $\cO_K$ a discrete valuation ring, every $\cO_K-$algebra of finite dimension is generated by a single element. We denote by $\Delta\in \cO_L$ an element such that
    \begin{equation*}
        \cO_L = \cO_K[\Delta].
    \end{equation*}
In the split case, we put $\Delta = (\Delta_1, \Delta_2)$. Since $\cO_L$ is the integral closure within $L$ of $\cO_K$, $\Delta$ satisfies a quadratic equation. We suppose it is 
    \begin{equation*}
        \Delta^2 = \tau_{\Delta}\Delta - \delta_{\Delta},
    \end{equation*}
    with $\tau_{\Delta}, \delta_{\Delta}\in \cO_K$.
    
In the nonsplit case we denote by $\pi$ a uniformizer of $\cO_L$ and by $\valpi$ the corresponding valuation of $L$. We write $Q$ for the cardinality of the residue field and $N_{L/K}$ for the norm of the extension $L/K$. We let $e$ and $f$ be the ramification and inertia degrees of the extension $L/K$. By definition, they satisfy
    \begin{align*}
        Q &= q^f,\\
        p\cO_L &= (\pi \cO_L)^e.
    \end{align*}
We also have $ef = 2$. It will be important to distinguish between the two possible nonsplit cases. We give the following
\begin{definition}
    Let $L$ be a quadratic field extension of $K$. We refer to the case $e = 2$ as the \textbf{ramified case}. When $e = 1$ we refer to it as the \textbf{unramified case}. 
\end{definition}

In the split case, we also denote by $N_{L/K}$ the norm of $L/K$. It is given by
\begin{equation*}
    N_{L/K}(x_1, x_2) = (x_1, x_2)(x_2, x_1) = x_1x_2\cdot(1, 1)\in K.
\end{equation*}
We will occasionally use the notation $\mathbf{1}=(1, 1)$ to distinguish between the unit of $\cO_K$ and that of $\cO_L$. We do this when the arguments require the use of the coordinates of $K\times K$. Otherwise, we use $1$ as the unit of $\cO_L$ regardless of the case. 

With these conventions at hand we can define the sequence of orders we will be interested in:
\begin{definition}
We define the \textbf{main sequence of orders} by
\begin{equation*}
    \cO_n := \cO_K[p^n\Delta], \; n\ge 0.
\end{equation*}
More specifically, $\cO_n\subseteq \cO_L$ consists of the elements of $\cO_L$ that can be written as
\begin{equation*}
    x + yp^{n}\Delta,
\end{equation*}
for some unique $x, y\in \cO_K$. Notice that the main sequence of orders is independent of the choice of $\Delta$.
\end{definition}
\begin{remark}
The orders $\cO_0, \cO_1,...$ are monogenic and thus satisfy $\cO_n = \check{\cO_n}$, as explained in remark \ref{isomorphictodual} in page \pageref{isomorphictodual}.  For the nonsplit case, $\cO_L$ is a discrete valuation ring and the isomorphism follows from the general theory of orders inside discrete valuation rings. 

In the split case, $\cO_L = \cO_K\times \cO_K$ is not a discrete valuation ring. The isomorphism nevertheless holds by a standard explicit computation of the isomorphism.  Indeed, it can be easily computed that
\begin{equation*}
    \check{\cO_n} = \cO_K\cdot v + \cO_K\cdot w,
\end{equation*}
where
\begin{equation*}
    v =\dfrac{1}{\Delta_2 - \Delta_1}(\Delta_2, 
    -\Delta_1),\;
    w = \dfrac{p^{-n}}{\Delta_2 - \Delta_1}(1, -1).
\end{equation*}
The $\cO_n-$module isomorphism $R:\check{\cO_n} \longrightarrow \cO_n$ is the one characterized by 
\begin{align*}
    R(v) &= p^{2n}\delta_{\Delta,}\\
    R(w) &= p^n\Delta.
\end{align*}
We leave the details to the reader.
\end{remark}
 We aim at finding the explicit polynomials related to the zeta functions of $\cO_n$ by proposition \ref{polynomial}. 
Let $n\ge 0$ be an integer. By definition, $\cO_n$ is a free $\cO_K-$module of rank $2$. We get immediately
\begin{proposition}
    Let $I\subseteq \cO_n$ be an ideal of $\cO_n$. We then have $I$ is a free $\cO_K$-module of rank at most $2$. Its $\cO_K-$rank is $2$ if and only if $[\cO_n: I]$ is finite.
\end{proposition}
We use the above proposition to give the following
\begin{definition}
    Let $I\subseteq \cO_n$ be an ideal. We say $I$ is a \textbf{rank 2 ideal} if its $\cO_K-$rank is exactly $2$. We denote by $\cI_n$ the set of rank 2 ideals of $\cO_n$. 
\end{definition}
Notice that the rank $2$ ideals of $\cO_n$ are precisely the ones that contribute to the zeta function of $\cO_n$. We now give the following
\begin{definition}
For each integer $n\ge 0$, define the map
\begin{equation*}
    T_n: \cI_n \longrightarrow \cI_{n + 1},
\end{equation*}
by
\begin{equation*}
    T_n(J) = pJ.
\end{equation*}
We call $T_n$ the \textbf{traveling map}.
\end{definition}
Due to the facts that  $p\cO_n\subseteq\cO_{n + 1}$ and  that multiplication by $p$ does not change the $\cO_K-$rank, we deduce that the image of $T_n$ indeed lies in $\cI_{n + 1}$. The goal of this section is to prove that this image consists exactly of the nonprincipal ideals in $\cI_{n + 1}$. 

The following is straightforward:

\begin{proposition}
For each $n\ge 1$ define $\psi_n: \cO_n\longrightarrow \cO_K/p^n\cO_K$ by
\begin{equation*}
    \psi_n(x + yp^n\Delta) = x \pmod{p^n\cO_K}.
\end{equation*}
Then $\psi_n$ is a surjective ring homomorphism with kernel $p^n\cO_0$ which induces an $\cO_n-$module structure on $\cO_K/p^n\cO_K$ given by
    \begin{equation*}
        (x + yp^n\Delta)\cdot\overline{b} = \psi_n(x + yp^n\Delta)\overline{b} = \overline{xb},
    \end{equation*}
    where the overline means class in $\cO_K/p^n\cO_K$.
\end{proposition}
The first isomorphism theorem now implies
\begin{corollary}\label{quotientequality}
In the context of the previous discussion,
\begin{equation*}
    \dfrac{\cO_n}{p^n\cO_0} \cong \dfrac{\cO_K}{p^n\cO_K}.
\end{equation*}
Notice the right hand side depends only on $K$ and not on the extension $L$.
\end{corollary}
 
Observe that $\cO_i$ is an $\cO_n-$module for $0\le i \le n$ with $\cO_n\subseteq \cO_i\subseteq \cO_0$. In fact there are no others, as we now show.

\begin{proposition}\label{onmodulesunsplit}
For $n\ge 0$, the only $\cO_n-$modules $M$ with
\begin{equation*}
    \cO_n\subseteq M\subseteq \cO_0
\end{equation*}
are precisely $\cO_0, \cO_1,...., \cO_n$.
\end{proposition}
\begin{proof}
    Let $\varepsilon_1, \varepsilon_2$ be an $\cO_K$ basis of $M$.  We must have that one of $\{1, \varepsilon_1\}$ or $\{1, \varepsilon_2\}$ is also an $\cO_K$ basis of $M$, because $1\in\cO_n\subseteq M$. Say it is $\{1, \epsilon_1\}$. 
    
    Write $\varepsilon_1 = x + y\Delta$, for some $x, y\in\cO_K$. Such $x$ and $y$ exist since $M\subseteq \cO_0 = \cO_L$. It follows easily that
    \begin{equation*}
        M = \cO_K[p^k\Delta],
    \end{equation*}
    where $k = \valp(y)$. Furthermore, $0\le k \le n$, for otherwise $p^n\Delta\in M$ would not be possible.
\end{proof}
We get the following important consequence from the previous proposition.
\begin{corollary}
    Let $n\ge 0$ and $\cO_n\subseteq M \subseteq \cO_0$ be an $\cO_n-$module. Then $M$ is a $\cO_n-$algebra.
\end{corollary}

Another easy fact states what the conductor between different orders is. More concretely, we have the following
\begin{proposition}\label{conductor}
The conductor ideal
\begin{equation*}
    (\cO_{n + m}: \cO_n) := \{ z\in \cO_{n + m} \;\mid\; z\cO_n\subseteq \cO_{n + m}\},
\end{equation*}
is given by
\begin{equation*}
    (\cO_{n + m}: \cO_n) = p^{m}\cO_n.
\end{equation*}
In particular, $(\cO_n: \cO_0) = p^n\cO_0$.
\end{proposition}
\begin{proof}
For $a + bp^{n + m}\Delta$ to satisfy
\begin{equation*}
    (a + bp^{n + m}\Delta)(x + yp^n\Delta)\in \cO_{n + m},
\end{equation*}
for all $x, y\in \cO_K$, it is enough that it does it for $x = 0$ and $y = 1$. In this case we have
\begin{equation*}
    (a + bp^{n + m}\Delta)(p^n\Delta) = -bp^{2n + m}\delta_{\Delta} + (ap^n +  bp^{n+m}\tau_{\Delta})p^n\Delta.
\end{equation*}
This implies $p^{n + m}|ap^n$, which is equivalent to $p^m|a$. Writing $a = p^ma_0$, we get
\begin{equation*}
    a + bp^{n + m}\Delta = p^m(a_0 + bp^n\Delta)\in p^m\cO_n,
\end{equation*}
as desired. Because all the steps are reversible, the double inclusion holds. This concludes the proof.
\end{proof}

We now study the units of $\cO_n$. Our aim is to give a characterization of them that is easy to use and verify.
\begin{proposition}
Let $n\ge 1$. Then $a + bp^n\Delta \in \cO_n$ is a unit of $\cO_L$ if and only if $a$ is a unit of $\cO_K$.
\end{proposition}
\begin{proof}
For $x\in \cO_L$, we know $x\in\cO_L^*$ if and only if $N_{L/K}(x)\in\cO_K^*$. From 
\begin{equation*}
    N_{L/K}(a + bp^n\Delta) = a^2 + p^nab\tau_{\Delta} + b^2p^{2n}\delta_{\Delta},
\end{equation*}
we deduce $N_{L/K}(a + bp^n\Delta)$ is a unit if and only if $a$ is a unit. This concludes the proof.
\end{proof}

\begin{proposition}  
We have $\cO_n^* = \cO_L^* \cap \cO_n$.
\end{proposition}
\begin{proof}
For $n= 0$ this is true because $\cO_0 = \cO_L$.

For $n\ge 1$, we clearly have $\cO_n^*\subseteq \cO_L^* \cap \cO_n$. We need to prove that if $a + bp^n\Delta$ is a unit of $\cO_L$, then its inverse also lies in $\cO_n$. The previous proposition implies $a$ is a unit of $\cO_K$. If $x + y\Delta$ is the inverse, then from the product equation 
\begin{equation}\label{inverseproduct}
    1 = (a + bp^n\Delta)(x + y\Delta) = (ax - bp^ny\delta_{\Delta}) + (ay + bp^nx + bp^ny\tau_{\Delta})\Delta,
\end{equation}
we deduce
\begin{equation*}
    ay + bp^{n}x + bp^{n}y\tau_{\Delta} = 0.
\end{equation*}
Hence $p^n|ay$. We deduce that $p^n|y$, because $a$ is a unit. That is, $x + y\Delta\in \cO_n$, as desired.
\end{proof}

Using the previous result we immediately get the following criterion, which we will use repeatedly:
\begin{theorem}
Let $n\ge 1$. Then $a + bp^n\Delta \in \cO_n$ is a unit if and only if $a$ is a unit of $\cO_K$.
\end{theorem}

The following definition will become useful later. 
\begin{definition}\label{typedef}
Let $x\in \cO_0 = \cO_L$ be any element. We define its \textbf{type} by
\begin{equation*}
    \varepsilon(x) = 
    \left\{
	\begin{array}{ll}
		\valpi(x)  & \mbox{in the nonsplit case }  \\
		(\valp(x_1), \valp(x_2)) & \mbox{in the split case }  
	\end{array}
\right.
\end{equation*}
where in the split case we have $x = (x_1, x_2)$. We define a \textbf{possible type} as an element of the image of $\varepsilon$.
\end{definition}
We will consider $\varepsilon(x)$ as a vector of one coordinate in the nonsplit case and of two coordinates in the split case. We now use the previous definition to give the following
\begin{definition}\label{representativedef}
Let $I\subseteq \cO_n$ be an ideal. We say $x\in I$ is a \textbf{representative} of $I$ if $\varepsilon(x) = \valpi(x)$ is minimal among $x\in I$ in the nonsplit case, and if 
\begin{align*}
    x_1 &= \min_{(x, y)\in I}( \valp(x))\\
    x_2 &= \min_{(x, y)\in I}( \valp(y))\\
\end{align*}
in the split case.
\end{definition}
The type is registering the valuation at each coordinate of the elements of $\cO_L$ as elements of $K$ or $K\times K$. A representative of an ideal is just an element of the ideal that has minimum possible valuation at each coordinate \textit{simultaneously}.

\begin{proposition}\label{repsexist}
Let $n\ge 0$ and $I\subseteq \cO_n$ a rank $2$ ideal, then $I$ has a representative. Furthermore, all its representatives are nonzero divisors of $\cO_L$. 
\end{proposition}
\begin{proof}
This is obvious in the nonsplit case because $\varepsilon(x) = \valpi(x)$ takes values in the nonnegative integers. We now address the split case.

Define $(x_1, y_1)\in I$ as one element such that
\begin{equation*}
    x_1 = \min_{(x, y)\in I}(\valp(x)),
\end{equation*}
and $(x_2, y_2)\in I$ as one element such that
\begin{equation*}
    y_2 = \min_{(x, y)\in I}(\valp(y)).
\end{equation*}
Both of these elements exists since, along a single coordinate, minimums of the valuations exist. 

If either of these elements has the other coordinate also a minimum then we are done. Otherwise we must have
\begin{equation*}
    \valp(x_2) > \valp(x_1), \valp(y_1) > \valp(y_2).
\end{equation*}
In this situation, define $(x_0, y_0) = (x_1 + x_2, y_1 + y_2)$. Since $I$ is closed under addition, we have $(x_0, y_0)\in I$. Furthermore,
\begin{equation*}
    \valp(x_0) = \valp(x_1 + x_2) = \min(\valp(x_1), \valp(x_2)) = \valp(x_1) = \min_{(x, y)\in I}(\valp(x)).
\end{equation*}
Notice we have used that the equality of the triangle inequality for valuations happens if the terms involved have different valuations, as is our case. Analogously,
\begin{equation*}
    \valp(y_0) = \min_{(x, y)\in I}(\valp(y)),
\end{equation*}
and this proves $(x_0, y_0)$ satisfies what we are looking for.

Furthermore, these valuations are nonnegative integers (as opposed to $\infty$) since otherwise the ideal would have all of its elements with $0$'s in the same coordinate. This would imply it does not have rank $2$, which contradicts our choice of $I$.
\end{proof}
Now we justify the name representative:
\begin{proposition}\label{reprentatives}
Let $n\ge 0$. For every rank $2$-ideal $I\subseteq \cO_n$ and every representative $x\in I$, there exists $0\le i \le n$ such that
\begin{equation*}
    I = x\cO_i.
\end{equation*}
Furthermore, $i$ only depends on $I$. 
\end{proposition}
\begin{proof}
The argument is the same for both split and nonsplit cases but we only write it for the split case.

Let $I\in\cI_n$.  The previous proposition implies there is a representative element $(x_0, y_0)\in I$ with
 \begin{align*}
    x_0 &= \min_{(x, y)\in I}(\valp(x)),\\
    y_0 &= \min_{(x, y)\in I}(\valp(y)). 
\end{align*}
For any element $(x, y)\in I$ consider 
 \begin{equation*}
     (z_1, z_2) = (xx^{-1}_0, yy^{-1}_0).
 \end{equation*}
Taking valuations we get
 \begin{align*}
     \valp(z_1) &= \valp(x) - \valp(x_0)\ge 0,\\
     \valp(z_2) &= \valp(y) - \valp(y_0)\ge 0,
 \end{align*}
 by definition of $(x_0, y_0)$. This implies
 \begin{equation*}
     (z_1, z_2)\in \cO_0,
 \end{equation*}
which is the same as saying
\begin{equation*}
    I\subseteq (x_0, y_0)\cO_0.
\end{equation*}
Because $I$ is an ideal, we also have
\begin{equation*}
    (x_0, y_0)\cO_n\subseteq I.
\end{equation*}
Hence, 
\begin{equation*}
    \cO_n\subseteq (x_0^{-1}, y_0^{-1})I \subseteq \cO_0.
\end{equation*}
Notice we can take the inverses since the entries are nonzero.
By proposition \ref{onmodulesunsplit}, we deduce there exists a unique $i$, with respect to $(x_0, y_0)$, such that
\begin{equation*}
    (x_0^{-1}, y_0^{-1})I = \cO_i,
\end{equation*}
that is, $I = (x_0, y_0)\cO_i$.

If $(x_1, y_1)$ and $(x_2, y_2)$ are two representatives for the same ideal with
\begin{equation*}
   I = (x_1, y_1)\cO_i = (x_2, y_2)\cO_j.
\end{equation*}
Without loss of generality suppose $i \le j$. There exists $o_i\in \cO_i$ and $o_j\in \cO_j$ with 
\begin{align*}
    (x_1, y_1) &= (x_2, y_2)o_j,\\
    (x_2, y_2) &= (x_1, y_1)o_i.
\end{align*}
This implies $o_io_j = \mathbf{1}$.
Since $\cO_j\subseteq \cO_i$ we have $o_j\in \cO_i$, which means that $o_j$ and $o_i$ are units of $\cO_i$. Hence
\begin{equation*}
    I = (x_1, y_1)\cO_i = (x_2, y_2)o_j\cO_i =  (x_2, y_2)\cO_i.
\end{equation*}
This implies $i = j$ for otherwise it contradicts the unicity of the index for the representative $(x_2, y_2)$.
\end{proof}
\begin{remark}
    The previous proof is classical in the theory of orders within discrete valuation rings. In the split case, $\cO_L$ is not a discrete valuation ring. In order to be able to carry out the argument, we needed to be able to minimize the valuation entrywise simultaneously. This is why we need the existence of representatives, guaranteed by proposition \ref{repsexist} above.
\end{remark}
We are now in position to prove the main structural result of the ideals in the main sequence of orders.
\begin{proposition}
Let $n\ge 1$ and $I\in \cI_n$ be a nonprincipal ideal. Then there exists a unique ideal $J\in \cI_{n - 1}$ such that $I = pJ$.
\end{proposition}
\begin{proof}
Let $I\in\cI_n$ be a nonprincipal ideal. By proposition \ref{reprentatives}, we have $I = x\cO_i$ for some representative $x\in I$ and a unique $0\le i \le n$. We know $i < n$, because $I$ is nonprincipal. In particular, since
\begin{equation*}
    x\cO_i \subseteq \cO_n,
\end{equation*}
$x$ lies in the conductor $(\cO_n: \cO_i)$. By proposition \ref{conductor}, we know $(\cO_n: \cO_i) = p^{n - i}O_i$. Write $x = p^{n - i}y$, for some $y\in \cO_i$. Substituting this in the expression of $I$, we get
\begin{equation*}
    I = x\cO_i = p(p^{n-i-1}y\cO_i).
\end{equation*}
Let $J := (p^{n-i-1}y\cO_i)$. Every ideal of $\cO_i$ that lies in $\cO_{n-1}$ is also an ideal of $\cO_{n-1}$. This implies $J$ is an $\cO_{n - 1}$-ideal. This concludes the proof of the existence. To notice it is unique, realize that $J = p^{-1}I$, which as a set is uniquely determined by $I$.
\end{proof}
From all above we immediately deduce
\begin{theorem}\label{travelingmap}
The traveling map $T_n$ is a bijection onto its image, which consists exactly of the nonprincipal rank $2$ ideals in $\cI_{n+1}$.
\end{theorem}


\section{Study of the Principal Part}\label{studyoftheprincipalpart}

We begin by giving the following
\begin{definition}
For $n\ge 0$ we define the \textbf{zeta function} of the order $\cO_n$, following definition \ref{zetafunction} in page \pageref{zetafunction}, by
\begin{equation*}
    \zeta_n(s) := \displaystyle\sum_{I\subseteq \cO_n}\dfrac{1}{[\cO_n: I]^s},
\end{equation*}
where the sum runs over all ideals $I\subseteq \cO_n$ of finite index, that is, over the ideals in $\cI_n$.

Furthermore, we denote by $\fP_n$ the set of principal ideals in $\cI_n$ and define
\begin{equation*}
    \zeta_n^P(s) := \displaystyle\sum_{I\in \fP_n}\dfrac{1}{[\cO_n: I]^s},
\end{equation*}
which we call \textbf{the principal part} of $\zeta_n(s)$.
\end{definition}

The purpose of the traveling map is to isolate the principal part of $\zeta_n$ away from the zeta function $\zeta_{n-1}$. More precisely, we have

\begin{theorem}\label{recurrencestatement}
For $n\ge 1$, the zeta functions of the orders $\cO_n$ satisfy the recurrence relation
\begin{equation*}
    \zeta_n(s) = \zeta_n^P(s) + q^{-s}\zeta_{n-1}(s).
\end{equation*}
\end{theorem}
\begin{proof}
It is clear that
\begin{equation*}
    \zeta_n(s) = \zeta_n^P(s) + \zeta^i_n(s),
\end{equation*}
where we define
\begin{equation*}
    \zeta^i_n(s) := \displaystyle\sum_{I\notin \fP_n} \dfrac{1}{[\cO_n:I]^s} .
\end{equation*}
If $J\subseteq \cO_{n - 1}$ then $pJ\subseteq \cO_n \subseteq \cO_{n-1}$ and $pJ\subseteq J \subseteq \cO_{n-1}$. 
Using the tower theorem for indices we get
\begin{equation*}
    [\cO_n:pJ] =\dfrac{ [\cO_{n-1}:pJ]}{[\cO_{n-1}:\cO_n]} = \dfrac{ [\cO_{n-1}:J][J:pJ]}{[\cO_{n-1}:\cO_n]}
\end{equation*}
We know that $[J:pJ] = q^2$ and $[\cO_{n-1}:\cO_n]=q$. Both of these equalities follow from elementary divisor theory. We then deduce
\begin{equation*}
   \zeta^i_n(s)     = \displaystyle\sum_{I\notin \fP_n} \dfrac{1}{[\cO_n:I]^s}
                     =\displaystyle\sum_{J\subseteq \cO_{n-1}} \dfrac{1}{q^s[\cO_{n-1}:J]^s}
                    = q^{-s}\zeta_{n-1}(s),
\end{equation*}
where we used the traveling map in the second equality.
\end{proof}

With this at hand the challenge becomes to find a usable expression for the principal part. In this direction we introduce the following

\begin{definition}
For $I\in\fP_n$ we denote by
\begin{equation*}
    R(I) := \{\alpha\in \cO_n \;\mid\; \alpha \cO_n = I\},
\end{equation*}
its set of representatives.
\end{definition}
We will organize ideals by the type their representatives have. Because different representatives of an ideal differ by a unit, we immediately get
\begin{proposition}\label{typeindeprep}
For $I\in\fP_n$ and $\alpha\in R(I)$, the type of $\varepsilon(\alpha)$ is independent of $\alpha$.
\end{proposition}

We now enhance the definition \ref{typedef} of type, given in page \pageref{typedef} above, to also apply for ideals in $\fP_n$. We have
\begin{definition}
Let $I\in\fP_n$. We define the \textbf{type of $I$} by
\begin{equation*}
    \varepsilon(I):= \varepsilon(x),
\end{equation*}
for any $x\in R(I)$.  We also define for every $x\in \cO_L$
\begin{equation*}
    \eta(x) = \min\{\valp(a) \mid a \mbox{ is a coordinate of } \varepsilon(x)\},
\end{equation*}
and write $\eta(I) = \eta(x)$ for any $x\in R(I)$. Notice that proposition \ref{typeindeprep} implies these definitions are well defined. 
\end{definition}
We now divide principal ideals into two different classes, in which the contribution to the zeta function is similar but not quite the same. We have
\begin{definition}
Let $\omega$ be a possible type. In the nonsplit case, we say $\omega$ is a \textbf{low type} if $\omega < ne$. In the split case, we say $\omega = (\omega_1, \omega_2)$ is a \textbf{low type} if $\omega_1 < n$ or $\omega_2 < n$.
Otherwise, we say $\omega$ is a \textbf{high type}. 
\end{definition}

\begin{definition}
Let $I\in\fP_n$. We say $I$ is a \textbf{low ideal} if $\varepsilon(I)$ is a low type. That is, if
\begin{equation*}
    \eta(I) < \left\{
	\begin{array}{ll}
		ne & \mbox{in the nonsplit case, }  \\
		n & \mbox{in the split case. }  
	\end{array}
\right.
\end{equation*}
Otherwise, we say $I$ is a \textbf{high ideal}. We will refer to $ne$ and $n$, respectively in the nonsplit and split cases, as the \textbf{ideal threshold} and denote if uniformly by $t_n$.
\end{definition}
The role of low and high ideals will become clearer as we move forward. For the moment, let us only say that high ideals contribute in all possible ways available, while low ideals have a more restricted contribution to the principal part of the zeta function.

We begin by showing what is the contribution of each ideal in $\fP_n$.  We have the following
\begin{definition}
Let $\omega$ be a possible type of an ideal in $\fP_n$. We define its \textbf{ type contribution} as
\begin{equation*}
    c(\omega):= \left\{
	\begin{array}{ll}
		f\omega & \mbox{in the nonsplit case, }  \\
		\omega_1 + \omega_2 & \mbox{in the split case. }  
	\end{array}
\right.
\end{equation*}
In here we are denoting $\omega = (\omega_1, \omega_2)$ in the split case, where the possible types are pairs of nonnegative integers.
\end{definition}
With this definition, the following becomes a straightforward application of the tower theorem for indices. We leave the details to the reader.
\begin{proposition}\label{contribution}
Let $I\in\fP_n$. The index $[\cO_n:I]$ is given by
\begin{equation*}
    [\cO_n:I] = q^{-c(\varepsilon(I))}.
\end{equation*}
\end{proposition}

Let $\omega$ be a possible type. In the nonsplit case $\omega$ is a nonnegative integer, while in the split case it is a pair of nonnegative integers. We define 
\begin{equation*}
   X_{\omega} :=  \{I\in\fP_n \mid\; \varepsilon(I) = \omega\}. 
\end{equation*}
Each ideal in $X_{\omega}$ contributes the same to the zeta function and grouping these terms together we deduce the principal part of the zeta function is
\begin{equation*}
    \zeta_n^P(s) = \displaystyle\sum_{\omega}\dfrac{| X_{\omega}|}{q^{c(\omega)}}.
\end{equation*}
The success of our analysis depends on whether we are able to compute the $|X_{\omega}|$ or not. 

The first step is to deal with the high types. The following proposition explains why this distinction is important.
\begin{proposition}\label{highabsorption}
Let $x\in \cO_0 = \cO_L$ be any element with $\eta(x) \ge t_n$, then $x\in \cO_n$.
\end{proposition}
\begin{proof}
Let $x\in \cO_0$ with $\eta(x)\ge t_n$. We claim $x$ is a multiple of $p^n$ in $\cO_0$. 

In the nonsplit case, $ \valpi(x) = \eta(x) \ge ne$. This implies
\begin{equation*}
   x = \pi^{\valpi(x)}u = \pi^{ne}\pi^{\valpi(x) - ne}u = p^n \left(\pi^{\valpi(x) - ne}u\right),
\end{equation*}
for some unit $u\in \cO_L$. On the other hand, for the split case
\begin{equation*}
   x = (x_1, x_2) = (p^{\valp(x_1)}u_1, p^{\valp(x_2)}u_2) = p^n(p^{\valp(x_1) -n}u_1, p^{\valp(x_2)-n}u_2).
\end{equation*}
for some unit $u_1, u_2\in \cO_K$. In any case, we conclude
\begin{equation*}
x\in p^n\cO_0 = (\cO_0: \cO_n) \subseteq \cO_n,
\end{equation*}
proving that $x\in \cO_n$.
\end{proof}

We now define, for every possible type $\omega$,
\begin{equation*}
    \cO_n^{[\omega]} := \displaystyle\bigsqcup_{\varepsilon(I) = \omega} R(I),
\end{equation*}
that is, the set of all representatives of principal ideals in $\fP_n$ that contribute $q^{-c(\omega)}$ to the principal part of the zeta function. The previous proposition implies that if $\omega$ is a high type, then
\begin{equation*}
    \cO_n^{[\omega]} = \left\{x\in \cO_0 \;|\; \varepsilon(x) = \omega\right\}.
\end{equation*}
In particular, $\cO_n^{[\omega]}\neq\emptyset$, that is, for high types there must be high ideals of type $\omega$. This is what we meant when we said above that \textit{high ideals contribute in all possible ways available}. The next proposition is now evident and states explicitly how many such ideals there are:
\begin{proposition}\label{highequivariant}
    Let $\omega$ be a high type. Define the map $\Psi_{\omega}: \cO_n^{[\omega]}\longrightarrow \cO_L^*$ given by 
    \begin{equation*}
        \Psi_{\omega}(x) = \pi^{-\omega}x 
    \end{equation*}
    in the nonsplit case and by
    \begin{equation*}
        \Psi_{\omega}(x_1, x_2) = (p^{-\omega_1}x_1, p^{-\omega_2}x_2) 
    \end{equation*}
    in the split case. Then $\Psi_{\omega}$ is an $\cO_n^*$-equivariant isomorphism. 
    
    In here, $\cO_n^{[\omega]}$ and $\cO_L^*$ are equipped with the natural actions of $\cO_n^*$ by left multiplication. As a consequence, the orbits of one are in correspondence with the orbits of the other, and so
    \begin{equation*}
        |X_{\omega}| = [\cO_0^*: \cO_n^*].
    \end{equation*}
\end{proposition}

This puts in solid ground the contribution of the high ideals. We now have to deal with the low ideals, which is a similar analysis, but we have to find how to deal with the fact that proposition \ref{highabsorption} is no longer true. The first step to do this is the following
\begin{proposition}
Let $I\in\fP_n$ be a low ideal and $x + yp^{n}\Delta \in R(I)$ a representative. Then $\valp(x)$ is independent of the representative of $I$ chosen and, furthermore, its value is smaller than $n$ and is equal to $\valp(a)$ for every coordinate of $\varepsilon(I)$. 
\end{proposition}
\begin{proof}
We will only do the computations for the split case as the nonsplit case is analogous. Realize there are no low ideals for $n = 0$, so we may assume $n \ge 1$. 

For the split case pick $\alpha\in R(I)$, say $\alpha = (\alpha_1, \alpha_2)$, and put $\varepsilon(I) = (l, m)$. Since $I$ is low one of $l$ or $m$ is smaller than $n$. Without loss of generality suppose $l < n$.

Write $\alpha = a\cdot\mathbf{1} + b\cdot(p^n\Delta)$, for some $a, b\in \cO_K$ which in coordinates becomes
\begin{equation*}
    \alpha_1 = a + bp^n\Delta_1,
    \alpha_2 = a + bp^n\Delta_2.
\end{equation*}
Suppose $\valp(a)\ge n$. Then
\begin{equation*}
    l   = \valp(a + bp^n\Delta_1) 
        \ge \min(\valp(a), n + \valp(b\Delta_1))
        \ge n,
\end{equation*}
which is a contradiction. We conclude $t := \valp(a) < n$, say $a = p^tu$ for some unit $u\in \cO_K$. Then, for $i = 1, 2$,
\begin{equation*}
    \alpha_i = p^tu + bp^n\Delta_i = p^t(u + bp^{n-t}\Delta_i),
\end{equation*}
and because $n - t > 0$ we have $bp^{n - t}\Delta_i$ is not a unit of $\cO_K$, and since $u$ is, then $u + bp^{n-t}\Delta_i$ is also a unit. We conclude,
$\valp(\alpha_i) = t = \valp(a)$, which proves independence, equality of entries and its value smaller than $n$, as desired.

\end{proof}
Due to this proposition we can make the following
\begin{definition}
Let $I\in\fP_n$ a low ideal. The common value of $\valp(x)$ for all the $x + yp^n\Delta\in R(I)$ is denoted by $\etap(I)$.
\end{definition}
As a consequence, contrary to the high types, we have that not all low types contribute to the zeta function. More precisely, we have
\begin{proposition}
Let $I\in\fP_n$ be a low ideal. Then, in the nonsplit case, we have
\begin{equation*}
    \eta(I) = e\etap(I),
\end{equation*}
and so the only possible types for low ideals are $0, e, 2e,..., (n -1)e$. On the other hand, for the split case
\begin{equation*}
    \varepsilon(I) = (\etap(I), \etap(I)),
\end{equation*}
and so the only possible types for low ideals are $(0, 0), (1, 1),..., (n -1, n - 1)$.
\end{proposition}
\begin{proof}
For the split case we have for $x + yp^n\Delta\in R(I)$ with $x = p^tu$ and $u\in \cO_K^*$ that 
 \begin{equation*}
     \eta(I) = \valpi(x + yp^n\Delta) = \valpi(p^t(u + yp^{n - t}\Delta)) = t\valpi(p) = te,
 \end{equation*}
 as desired. Notice that for the last equality we have used that $t < n$, by the previous proposition, and hence that $u + yp^{n - t}\Delta$ is a unit of $\cO_L$. This implies that the possible types of low ideals, which are precisely the values of $\varepsilon(I) =\eta(I)$, can only be the multiples of $e$ from $0$ to $(n - 1)e$.

For the split case, the previous proposition implies the entries of $\varepsilon(I)$ are equal and smaller than $n$. Hence, the possible types are $(0, 0),..., (n - 1, n - 1)$, as claimed.

\end{proof}
Notice in both cases we have a \textit{linear behaviour} for the types of low ideals, which is a remarkable fact for the split case since, a priori, there are several other possible low types that end up never appearing.
Similar to proposition \ref{highequivariant}, we now have a group action for the low types:
\begin{proposition}
Let $\omega$ be one of the types of the previous proposition (i.e. the possible ones for low ideals). 
Define $\Psi_{\omega}: \cO_n^{[\omega]}\longrightarrow \cO_{n - d}^*$ by
    \begin{equation*}
        \Psi_{\omega}(x) = p^{-d}x.
    \end{equation*}
where $d = \eta(\omega)/e$ in the nonsplit case and $d = \eta(\omega)$ in the split case. Then $\Psi$ is a $\cO_n^*$-equivariant isomorphism between the $\cO_n^*$-spaces $\cO_n^{[\omega]}$ and $\cO_{n - d}^*$, where in both spaces the action is by left multiplication. As a consequence, we have
\begin{equation*}
    |X_{\omega}| = [\cO_{n - d}^*:\cO_n^*] = \dfrac{[\cO_{0}^*:\cO_n^*]}{[\cO_{0}^*:\cO_{n - d}^*]}
\end{equation*}
\end{proposition}
\begin{proof}
For the nonsplit case, an ideal $I$ whose representatives have $\eta(I) = \omega$ satisfies
\begin{equation*}
    \eta(I) = de, \etap(I) = d.
\end{equation*}
On the other hand, for the split case we have $\varepsilon(x) = (d, d)$. Hence, in both cases, we can write any representative as
\begin{equation*}
    p^du + yp^n\Delta,
\end{equation*}
for some unit $u\in \cO_K^*$. Since $d < n$ we can factor the $p^d$ and get
\begin{equation*}
    p^d(u + yp^{n - d}\Delta),
\end{equation*}
and since $u$ is a unit of $\cO_K$, then $u + yp^{n - d}$ is a unit of $\cO_{n - d}$. Since all of these steps are reversible, we get that $\Psi_d(\cO_n^{[\omega]}) = \cO_{n - d}^*$.

That it is equivariant is immediate and that all orbits count follows from the fact that the representatives are nondivisors of $0$ and we can appeal to proposition \ref{contribution}.

\end{proof}

Now we are ready to prove the main result of this section, which is the following
\begin{proposition}\label{recurrenceshape}
The principal part of the zeta functions of the orders $\cO_n$ satisfies
\begin{equation*}
        \zeta_n^P(s) = \displaystyle\sum_{d = 0}^{n - 1}\dfrac{[\cO_{0}^*:\cO_n^*]}{[\cO_{0}^*:\cO_{n - d}^*]}\cdot\dfrac{1}{q^{2ds}} + \dfrac{[\cO_{0}^*:\cO_n^*]}{q^{2ns}V(s)}.
\end{equation*}
In here $V(s)$ is the factor appearing in proposition \ref{polynomial} on page \pageref{polynomial}.
\end{proposition}
\begin{proof}
Collecting the contributions of both high and low ideals we will obtain this result. For the nonsplit case the high ideals contribute
\begin{equation*}
    \displaystyle\sum_{j\ge 0} \dfrac{|X_{ne + j}|}{q^{f(ne + j)s}}
    = \displaystyle\sum_{j\ge 0} \dfrac{[\cO_{0}^*:\cO_n^*]}{q^{2ns + fjs}} =\dfrac{[\cO_{0}^*:\cO_n^*]}{q^{2ns}(1 - q^{-fs})}.
\end{equation*}
For the low ideals we only get contributions for $0, e,..., (n - 1)e$, and for each one of them we have $|X_{de}|$ such terms. Hence, we obtain the low ideals contribute
\begin{equation*}
    \displaystyle\sum_{d = 0}^{n - 1} \dfrac{|X_{de}|}{q^{f(de)s}}
    = \displaystyle\sum_{d = 0}^{n - 1}\dfrac{|X_{de}|}{q^{2ds}}
    = \displaystyle\sum_{d = 0}^{n - 1}\dfrac{[\cO_{0}^*:\cO_n^*]}{[\cO_{0}^*:\cO_{n - d}^*]}\cdot\dfrac{1}{q^{2ds}}
\end{equation*}
Notice that we have used $fe = 2$.
On the other hand, for the high ideals in the split case, we get

\begin{equation*}
    \displaystyle\sum_{\eta(I)\ge n}\dfrac{1}{[\cO_n:I]^s}
    = \displaystyle\sum_{l, m\ge n}\displaystyle\sum_{\varepsilon(I) = (l, m)}\dfrac{1}{[\cO_n:I]^s}
    = \dfrac{[\cO_L^*:\cO_n^*]}{q^{2ns}}\displaystyle\sum_{i, j\ge 0}\dfrac{1}{q^{(i + j)s}}
    =\dfrac{[\cO_L^*:\cO_n^*]}{q^{2ns}(1 - q^{-s})^2}.
\end{equation*}
For the low ideals we have the possible types are only $(0,0),..., (n - 1, n-1)$. Hence, the contribution is
\begin{equation*}
    \displaystyle\sum_{d = 0}^{n - 1} \dfrac{|X_{(d, d)}|}{q^{(d + d)s}}
    = \displaystyle\sum_{d = 0}^{n - 1}\dfrac{|X_{(d, d)}|}{q^{2ds}}
    = \displaystyle\sum_{d = 0}^{n - 1}\dfrac{[\cO_{0}^*:\cO_n^*]}{[\cO_{0}^*:\cO_{n - d}^*]}\cdot\dfrac{1}{q^{2ds}}.
\end{equation*}
Finally, by inspection of the two cases, we see  $V(s)$ coincides with $(1 - q^{-fs})$ in the nonsplit case and with $(1 - q^{-s})^2$ in the split one.
\end{proof}


\section{Solution of the Recurrence}\label{solutionoftherecurrence}

We have seen that the principal part depends on the indices $[\cO_0^*: \cO_n^*]$. We now compute the values of these indices. 

The following formula is classical in the general theory of orders but in the case we are is particularly simple, since we do not have to deal with class numbers. We refer the reader to \cite{Neu99} (see theorems 11 and 12 in chapter 1, section 12) or \cite{Sands1991} to see a proof of the general case which implies our case (the proof also works in the split case).

\begin{proposition}
The following formula holds
\begin{equation*}
    [\cO_0^*:\cO_n^*] = \dfrac{\abs{\left(\nicefrac{\cO_0}{p^n\cO_0}\right)^*}}{\abs{\left(\nicefrac{\cO_n}{p^n\cO_0}\right)^*}}
\end{equation*}
\end{proposition}

We have now reached the point where it is desirable to make the distinction between the two possible nonsplit cases. 
\begin{proposition}\label{valuesindices}
The indices of the units subgroups satisfy for $n\ge 1$,
\begin{equation*}
    [\cO_0^*:\cO_n^*] = \left\{
    \begin{array}{ll}
        q^n & \mbox{in the ramified case,}\\
        (q+1)q^{n-1} & \mbox{in the unramified case,}\\
        (q-1)q^{n-1} & \mbox{in the split case.}\\
    \end{array}
    \right.
\end{equation*}
And, of course, $[\cO_0^*:\cO_0^*] = 1$.
\end{proposition}
\begin{proof}
All the computations are analogous so we only show the nonsplit case so that we can notice clearly how the ramification $e$ and intertia $f$ play a role. 

We have that $\nicefrac{\cO_0}{p^n\cO_0} = \nicefrac{\cO_L}{\pi^{ne}\cO_L}$, and we have the well known isomorphisms
\begin{align*}
    \left(\nicefrac{\cO_L}{\pi^{ne}\cO_L}\right)^* &\cong \nicefrac{\cO_L^*}{U^{(ne)}_L},\\
    \nicefrac{U_L^{(m)}}{U^{(m + 1)}_L} &\cong \nicefrac{\cO_L}{\pi \cO_L}, m\ge 1,
\end{align*}
where $U_L^{(m)}$ are the higher unit groups. Furthermore, we also know that $O_L^* = \mu_{Q - 1}\times U^{(1)}_L$, where $\mu_{Q-1}$ is the group of $(Q - 1)$ roots of unity. From this we get
\begin{equation*}
    \nicefrac{\cO_L^*}{U_L^{(1)}}\cong \mu_{Q - 1}.
\end{equation*}
Now we can iterate and obtain
\begin{equation*}
     \abs{\left(\nicefrac{\cO_L}{\pi^{ne}\cO_L}\right)^*} 
     = \abs{\nicefrac{\cO_L^*}{U^{(ne)}_L}}
     =  \abs{\nicefrac{\cO_L^*}{U^{(1)}_L}}\abs{\nicefrac{U_L^{(1)}}{U^{(2)}_L}}\cdots \abs{\nicefrac{U_L^{(ne - 1)}}{U^{(ne)}_L}}
     = (Q - 1)Q^{ne - 1}.
\end{equation*}
Corollary \ref{quotientequality} on page \pageref{quotientequality} states
\begin{equation*}
    \nicefrac{\cO_n}{p^n \cO_0} \cong \nicefrac{\cO_K}{p^n \cO_K}.
\end{equation*}
Hence,
\begin{equation*}
    \left(\nicefrac{\cO_n}{p^n \cO_0}\right)^* \cong \left(\nicefrac{\cO_K}{p^n \cO_K}\right)^*.
\end{equation*}
By the same iterative argument we thus conclude
\begin{equation*}
    \abs{\left(\nicefrac{\cO_n}{p^n \cO_n}\right)^*} = (q - 1)q^{n - 1}.
\end{equation*}
Now we specialize in each of the cases we have:
\begin{description}
    \item[Ramified, $f = 1$:] Then $e = 2$ and, using the previous proposition result, and recalling $Q = q^f = q$, we get
    \begin{equation*}
        [\cO_0^*: \cO_n^*]  = \dfrac{(Q - 1)Q^{2n - 1}}{(q - 1)q^{n - 1}}
                        = \dfrac{(q - 1)q^{2n - 1}}{(q - 1)q^{n - 1}}
                        = q^n,
    \end{equation*}
    and notice that this same formula works if $n = 0$.
    \item[Unamified, $f = 2$:] In this situation $Q = q^2$ and $e = 1$, then
    \begin{equation*}
        [\cO_0^*: \cO_n^*]  = \dfrac{(Q - 1)Q^{n - 1}}{(q - 1)q^{n - 1}}
                        = \dfrac{(q^2 - 1)q^{2(n - 1)}}{(q - 1)q^{n - 1}}
                        = (q + 1)q^{n - 1}.
    \end{equation*}
\end{description}
This concludes the proof.
\end{proof}

We have found before that the recurrence relation in each case is
\begin{equation*}
    \zeta_n(s) = \zeta_n^P(s) + q^{-s}\zeta_{n - 1}(s).
\end{equation*}
We are now ready to solve this equations explicitly. We begin by finding the initial condition:
\begin{proposition}
The zeta function of the order $\cO_0 = \cO_L$ is in the nonsplit case
\begin{equation*}
    \zeta_0(s) = \dfrac{1}{1 - q^{-fs}},
\end{equation*}
and
\begin{equation*}
    \zeta_0(s) = \dfrac{1}{(1 - q^{-s})^2}.
\end{equation*}
in the split case. In particular, in each case, $\zeta_0(s) = V(s)^{-1}$, where $V(s)$ is the factor appearing in proposition \ref{polynomial} at page \pageref{polynomial}.
\end{proposition}
\begin{proof}
We know that for $n = 0$ all ideals are high. We have computed before, in the proof of proposition \ref{recurrenceshape} on page \pageref{recurrenceshape}, that the high ideals contribute
$\dfrac{[\cO_{0}^*:\cO_n^*]}{q^{2ns}V(s)}$. In our present case this means
\begin{equation*}
    \zeta_0(s) = \zeta_0^P(s) = \dfrac{1}{V(s)},
\end{equation*}
and $V(s)$ is $1 - q^{-fs}$ for the nonsplit case and $(1 - q^{-s})^2$ for the split one.
\end{proof}

We finally get
\begin{theorem}\label{solutionrecurrence}
For each $n\ge 0$ define the following polynomials:
\begin{align*}
        R_n(X) &= 1 + qX^2 + q^2X^4 + ... + q^nX^{2n},
\end{align*}
and for $n\ge 1$ define
\begin{align*}
    U_n(X) &= (1 + X)R_{n - 1}(X) + q^nX^{2n},\\
    S_n(X) &= (1 - X)R_{n - 1}(X) + q^nX^{2n}.
\end{align*}
Finally, also put $U_0(X) = S_0(X) = 1$. Explicitly, these polynomials are
\begin{align*}
    R_n(X) &= 1 + qX^2 + q^2X^4 + ... + q^nX^{2n},\\
    U_n(X) &= 1 + X + qX^2 + qX^3 + ... + q^{n - 1}X^{2n - 2} + q^{n - 1}X^{2n - 1} + q^nX^{2n},\\
    S_n(X) &= 1 - X + qX^2 - qX^3 + ... + q^{n - 1}X^{2n - 2} - q^{n - 1}X^{2n - 1} + q^nX^{2n}.
\end{align*}Then the solution of the ramified, unramified and split case recurrence, respectively, satisfy
\begin{align*}
    (1 - q^{-s})\zeta_n(s) &= R_n(q^{-s}),\\
    (1 - q^{-2s})\zeta_n(s) &= U_n(q^{-s}),\\
    (1 - q^{-s})^2\zeta_n(s) &= S_n(q^{-s}).
\end{align*}
\end{theorem}
\begin{proof}
The three cases are solved by induction using as base case the initial condition given in the previous proposition. We omit the details as they are straightforward. 
\end{proof}

\bibliographystyle{acm}
\bibliography{bibliography}

\end{document}